\documentclass[10pt]{article}
\usepackage{amsmath}
\usepackage{amssymb}

\usepackage{amsthm}
\newtheorem{theorem}{Theorem}
\newtheorem{lemma}[theorem]{Lemma}
\newtheorem{proposition}[theorem]{Proposition}
\newtheorem{corollary}[theorem]{Corollary}

\theoremstyle{definition}
\newtheorem{definition}[theorem]{Definition}
\newtheorem{remark}[theorem]{Remark}
\newtheorem{example}[theorem]{Example}

\newcommand{\reals}{\mathbb{R}}
\newcommand{\complex}{\mathbb{C}}
\newcommand{\quaternions}{\mathbb{H}}
\newcommand{\R}{\reals}
\newcommand{\C}{\complex}
\newcommand{\HH}{\quaternions}

\newcommand{\M}{\mathcal{M}}
\newcommand{\sdet}{\mathop{{\rm Sdet}}}
\newcommand{\rank}{\mathop{{\rm rank}}}
\newcommand{\diag}{{\rm diag}}

\newcommand{\I }{\mathbf{i}}
\newcommand{\J }{\mathbf{j}}
\newcommand{\K}{\mathbf{k}}

\newcommand{\id}{\,\mathrm{Id}}
\newcommand{\sgn}{\mathrm{sgn}}

\begin{document}

\title{A topological approach to left eigenvalues\\ of quaternionic matrices\thanks{Partially supported by FEDER and Research Project MTM2008-05861 MICINN Spain.}}

\author{
E.~Mac\'{\i}as-Virg\'os
\thanks{Facultade de Matem\'aticas, Universidade de Santiago de Compostela, 15782 - Spain {\tt quique.macias@usc.es}} 
and M.~J.~Pereira-S\'aez\thanks{ Facultade de Econom\'{\i}a e Empresa, Universidade da Coru\~na, 15071 - Spain {\tt maria.jose.pereira@udc.es}} \\
}

\markboth{E.\ Mac\'ias-Virg\'os and M.\ J.\ Pereira-S\'aez}{Linear and Multilinear Algebra}

\maketitle

\begin{abstract}
It is known that a $2\times 2$ quaternionic matrix has one, two or an infinite number of left eigenvalues, but the available algebraic proofs  are  difficult to generalize to higher orders. In this paper a different point of view is adopted by  computing the topological degree of a characteristic map associated to the matrix and discussing the rank of the differential. The same techniques are extended to $3\times 3$ matrices, which are still lacking a complete classification.
\end{abstract}



%%%%%%%%%%%%%%%

\section{INTRODUCTION}
In 1985, Wood   \cite{WOOD} proved that any $n\times n$ quaternionic matrix $A$  has at least one {\em left eigenvalue}, that is  a quaternion $\lambda\in \quaternions$ such that the matrix $A-\lambda \id$ is  singular. However, even for matrices of small size the left spectrum is not fully understood yet (see Zhang's papers \cite{ZHANG1997,ZHANG2007} for a survey). For instance, it was only in 2001 when Huang and So \cite{HUANGSO2001} proved  that a $2\times 2$ matrix may have one, two or an infinite number of left eigenvalues; a different proof was  presented by the authors in \cite{MP1}. While Wood used topological techniques, namely homotopy groups, the  two latter papers are of algebraic nature and seemingly difficult to generalize for $n>2$.

In this article we try a different approach. The basic ideas will be those of characteristic map, linearization and topological degree. In the simplest case $n=2$ we associate to each matrix $A$ a polynomial $\mu_A$ whose roots are the left eigenvalues; computing the rank of its differential allows us to classify the different types of spectra. 

In the second part of the paper we extend those techniques to $3\times 3$ matrices. This time the characteristic map may not be a polynomial but a rational map, usually with a  point of discontinuity. This seems to require the use of a local version of degree, although a closer look allows us to reduce the problem to the global theory. In particular, this gives a new proof of the existence of left eigenvalues.

Unlike the $2\times 2$ case, where the linear equations that appear correspond to the well known Sylvester equation, the $3\times 3$ situation is much more complex. Then for  $n=3$ a complete classification of spectra is still unknown. Nevertheless our method allows to deal with specific examples and opens the way for a better understanding of the general case.

The paper is organized as follows. In Section \ref{TOPOL} we recall some topological and algebraic preliminaries. Although our ideas are closely related to the theory of quasideterminants \cite{GELFAND2005}, we have preferred a development based on Study's determinant which parallels the commutative setting. Section \ref{CHARSECT} is devoted to a notion of  {\em characteristic map} for the left eigenvalues of a quaternionic matrix $A$, that is a map $\mu_A\colon \HH \to \HH$ such that $\mu_A(\lambda)=0$ if and only if the matrix $A-\lambda\id$ is not invertible. In Section \ref{ORDERTWO} we give a complete classification of the left spectra of $2\times 2$ matrices, depending on the rank of the characteristic map. In Section \ref{CHARMAP3x3} we prove that any $3\times 3$ matrix $A$ has a characteristic map $\mu_A$ which is either a polynomial (when some of the entries outside the diagonal is null) or a rational function with a distinguished point $\pi_A$ called its {\em pole}. When $\mu_A$ is continuous it has topological degree $3$. However, $\pi_A$ may be a point of discontinuity; in this case the matrix $B=A-\pi_A\id$ turns out to be invertible and we prove that  $B$ and $B^{-1}$ have diffeomorphic spectra and that $B^{-1}$ admits a polynomial characteristic map. The last Section offers several illustrative examples.

\section{PRELIMINARIES}\label{TOPOL}
\subsection{Topological degree}
The topological degree (or Brouwer degree) of a map can be defined by techniques either from algebraic topology \cite{DIEUDONNE} or from functional analysis  \cite{AMBROSETTI,DEIMLING,PUMWAL}.
We want to apply the following global result (cf. \cite[p.~101]{MADSENTORNEHAVE}):
\begin{theorem}\label{REGULAR}
  Let $M$ be a connected closed oriented manifold, let $\mu\colon M \to M$ be a differentiable map of degree $k$. Let $m\in M$ be a regular value such that the differential $\mu_{*\lambda}$ preserves the orientation for any $\lambda$ in the fiber $\mu^{-1}(m)$. Then $\mu^{-1}(m)$ is a finite set  with $k$ points.
\end{theorem}

Sometimes one has to deal with a local notion of degree. The main result  is the following one \cite[p.~38]{AMBROSETTI}.

\begin{theorem} Let $\Omega$ be a bounded open set in $\R^n$. Let $\mu \colon \overline{\Omega} \to \R^n$ be  a map which is  continuous on the closure $\overline{\Omega}$ and differentiable on $\Omega$. Suppose that $0$ is a regular value of $\mu$ and that $0\notin \mu(\partial\Omega)$. Then
$$\deg(\mu,\Omega,0)=
\sum\limits_{\lambda\in\mu^{-1}(0)}{
\sgn[J_\mu(\lambda)]
}$$
where we denote by $J_\mu$ the Jacobian of $\mu$.
\end{theorem}

A well known consequence is that if $\deg(\mu,\Omega,0)\neq 0$  then the equation $\mu(\lambda)=0$ has at least one solution in $\Omega$. In fact, for maps of the sphere into itself, the  {\em existence} of solutions only depends on the {\em continuity}, by the following result \cite[Ch. VIII, Ex. 2.5, p.~191]{MASSEY}:

\begin{proposition} Let $\mu\colon S^4\to S^4$ be a continuous map whose degree is nonzero. Then $\mu$ is surjective.
\end{proposition}

\subsection{Linearization}
We consider the space of quaternions $\quaternions$ as a differentiable manifold diffeomorphic to $\reals^4$. Then the differential $\mu_{*\lambda}\colon \quaternions \to \quaternions$ at the point $\lambda\in\quaternions$  of the differentiable map $\mu \colon \quaternions \to \quaternions$  can be computed by means of the   formula
$$\mu_{*\lambda}(X)={\frac{d}{dt}}_{\vert t=0}{\mu(\lambda+tX)}=\lim\limits_{t\to 0}{\frac{1}{t}\left(\mu(\lambda+tX)-\mu(\lambda)\right)}.$$

\begin{lemma}\label{ReglasDerivacion}
\begin{enumerate}
\item
Let $f,g\colon \quaternions \to \quaternions$ be two differentiable maps. Then the differential of the product is given by
$$(f\cdot g)_{*\lambda}(X)=f_{*\lambda}(X)\cdot g(\lambda)+f(\lambda)\cdot g_{*\lambda}(X);$$
\item
assume that $f(\lambda)\neq 0$ for all $\lambda\in \quaternions$. Let $1/f\colon \quaternions \to \quaternions$ be the map given by $(1/f)(\lambda)=f(\lambda)^{-1}$.
Then
$$(1/f)_{*\lambda}(X)=-f(\lambda)^{-1}f_{*\lambda}(X)f(\lambda)^{-1}.$$
\end{enumerate}
\end{lemma}

\subsection{Sylvester equation}\label{SYLVESTER}
Let $P,Q,R\in \quaternions$ be three quaternions. We are interested (see formula (\ref{DIFER})) in the rank of $\reals$-linear maps $\Sigma\colon \quaternions \to \quaternions$ of the form
$\Sigma(X)=PX+XQ$. The equation $\Sigma(X)=R$ has been widely studied, sometimes under the name of Sylvester equation \cite{JOHNSON}.

\begin{lemma}
\begin{enumerate}
\item
Let $P=t+x\I+y\J+z\K$. Then the matrix associated  to the $\reals$-linear map $X\mapsto PX$ with respect to the basis $\{1,\I ,\J ,\K \}$  is
$L(P)=\Re(P)\id+A(P)$, where $\Re(P)$ is the real part of $P$ and
$$A(P)=\begin{bmatrix}
0 & -x & -y & -z \cr
x & 0 & -z & y \cr
y & z & 0 & -x \cr
z & -y & x & 0
 \end{bmatrix};$$
\item
analogously,  if $Q=s+u\I+v\J+w\K$ then the matrix associated to the right translation $X\mapsto XQ$  is
$R(Q)=\Re(Q)\id+B(Q)$, where
$$B(Q)=\begin{bmatrix}
0 & -u & -v & -w \cr
u & 0 & w & -v \cr
v & -w & 0 & u \cr
w & v & -u & 0\cr
 \end{bmatrix}.$$
\end{enumerate}
\end{lemma}
Next Theorem is a reformulation of the results by Janovsk\'a and Opfer in  \cite{JANOVSKAOPFER2008},  see also \cite{GEORGIEV}.
\begin{theorem}\label{RANGO}
\begin{enumerate}
\item
 The rank of  $\Sigma$ is even, namely $0$, $2$ or $4$;
 \item
 $\rank \Sigma < 4$ if and only if $P$ and $- Q$ are similar quaternions, i.e. they have the same norm  and the same real part;
 \item
 $\rank \Sigma =0$ if and only if $P$ is a real number and $Q=-P$. \end{enumerate}
\end{theorem}

\begin{proof}
The  matrix associated to $\Sigma$  is
$$J=\begin{bmatrix}
t+s&-x-u&-y-v&-z-w\cr
x+u&t+s&-z+w&y-v\cr
y+v&z-w&t+s&-x+u\cr
z+w&-y+v&x-u&t+s\cr
\end{bmatrix}.$$
Since
\begin{eqnarray}\label{DET}
 \det J=
  (t+s)^4 + 2  (t + s)^2 (x^2 + y^2 + z^2+u^2 + v^2 + w^2) +&&\\ \nonumber
   (x^2 + y^2 + z^2-u^2 - v^2 - w^2 )^2&\geq&0,
  \end{eqnarray}
the matrix $J$ has rank $4$ excepting when $t+s=0$ and $x^2+y^2+z^2=u^2+v^2+w^2$. In this case $J$ is skew-symmetric, hence its rank is even. If $\rank\Sigma=0$ it follows that $x=y=z=0$ and $u=v=w=0$.
%Alternatively, notice that when $\rank \Sigma =2$ the homogeneous equation $\Sigma(X)=0$ has a non trivial solution $X\neq 0$, that is $P=-XQX^{-1}$, which is well known to be equivalent to the conditions $\Re(P)=-\Re(Q)$ and $\vert P\vert = \vert Q \vert$.
\end{proof}

\subsection{Resolution of arbitrary linear equations}\label{ResEcLineales}
More generally, let us consider a linear equation of the form
\begin{equation}\label{ECLINCUAT}
P_1 X Q_1+\cdots+P_n X Q_n=R, \quad \mbox{with\ }P_i,Q_i, R\in\quaternions.
\end{equation}
For each  bilateral term $X\mapsto PXQ$, the matrices $L(P)$ and $R(Q)$ commute, because $P(XQ)=(PX)Q$. Then  $A(P)$ and $B(Q)$ commute too. This implies that  the quaternionic linear equation (\ref{ECLINCUAT}) is equivalent to the real linear system $MX=R$, where $X,R\in\R^4$
and $M$ is the $4\times 4$ real  matrix
$$M=\sum_i{L(P_i)R(Q_i)}=\sum_i{a_ib_i}\id + \sum_i{(a_iB_i+b_iA_i)}+\sum_i{A_iB_i}$$
with $a_i=\Re(P_i)$, $b_i=\Re(Q_i)$, $A_i=A(Q_i)$ and $B_i=B(Q_i)$.
Contrary to the case $n=2$, when $n\geq 3$ the rank of $M$ may  be odd.

\begin{example} {\rm The rank of the matrix associated to the bilateral linear equation $\K X+X(2-\I)-2\J X\J=0$ is $3$.
% y tiene como autovalores $\{4,2,2,0\}$.
}\end{example}

\subsection{Determinants}\label{Sdet}
It is possible to generalize to the quaternions the norm $\vert \det \vert$ (that is, with real values) of the complex determinant.
Papers \cite{ASLAKSEN1996,COHENLEO2000}  are surveys of the general theory of quaternionic determinants. For the relationship between Study's determinant and quasideterminants see \cite[pp.~76--85]{GELFAND2005}.
\begin{definition}\label{STUDY}
Let the quaternionic matrix $A\in\M_{n\times n}(\HH)$ be decomposed as
 $A=X+\J Y$ with $X,Y \in\M_{n\times n}(\C)$.
We shall call  \emph{Study's determinant} of $A$ the non-negative real  number
$$\sdet(A) := (\det c(A))^{1/2},$$
where $c(A)$ is the complex matrix
$%\label{forma compleja}
c(A)=\begin{bmatrix}
X & -\overline{Y} \cr
Y & \overline{X}\cr
\end{bmatrix}
\in\mathcal{M}_{2n\times 2n}(\mathbb{C}).$
\end{definition}

\begin{remark} {\rm Up to the exponent, this is the same determinant which appears in  Theorem 8.1 of \cite{ZHANG1997}, as well as others considered in  \cite{ASLAKSEN1996}.
We have normalized the exponent to $1/2$ in order to ensure that $\sdet(D)=\vert q_1\ldots q_n\vert $  for diagonal matrices
$D=\diag(q_1,\ldots q_n)$.}
\end{remark}

\begin{proposition}[\cite{COHENLEO2000}]\label{propiedadessdet}
$\sdet$ is the only functional that verifies the  properties:
\begin{enumerate}
\item $\sdet(AB)=\sdet(A)\cdot\sdet(B)$;
\item if $A$ is a complex matrix then $\sdet(A)=\vert\det(A)\vert$.
\end{enumerate}
\end{proposition}

The following immediate consequences are very useful for computations.

\begin{corollary}\label{PropSdet2}
\begin{enumerate}
\item $\sdet(A)>0$ if and only if the matrix $A$ is invertible;
\item let $A$ and $B=PAP^{-1}$ be similar matrices, then $\sdet(A)=\sdet(B)$;
\item
$\sdet(A)$ does not change when a (right)  multiple of one column  is added to another column;
\item
$\sdet(A)$ does not change when a (left) multiple of one row is added to another row;
 \item
$\sdet(A)$ does not change when two columns (or two rows) are permuted.
\end{enumerate}
\end{corollary}

We shall need the following result too (we have not found it explicitly in the literature):

\begin{proposition}\label{sdet por cajas}
For any matrix with two boxes $A,B$ of order $m$ and $n$ respectively, it holds that
$
\sdet\begin{bmatrix}
A & 0 \cr* & B\cr
\end{bmatrix}=\sdet(A)\cdot\sdet(B)
$.
\end{proposition}

\subsection{Jacobi identity}\label{identidad Jacobi}
Let $C$ be a complex $n\times n$ matrix.
Let $I=\{i_1,\dots,i_p\}$ and $J=\{j_1,\dots,j_p\}$ be two subsets of $\{1,\dots,n\}$  with the same size $p$.  Let us denote by $C_{I,J}$ the submatrix formed by the rows with index in $I$ and the columns with index in $J$. On the other side, let us denote by $C^{I,J}$ the complementary matrix obtained by suppressing the rows in  $I$ and the columns in $J$.

The following {\em Jacobi identity}  is attributed to Kronecker in \cite{NANSON1900}.
\begin{lemma}\label{KRON} Assume that the complex matrix $C$ is invertible. Then
$$\det{(C^{-1})_{I,J}}=(-1)^{I+J}\det{C^{J,I}}/\det C,$$
where $I+J$ means $i_1+\cdots+i_p+j_1+\cdots+j_p$.
\end{lemma}

A generalization to quasideterminants appears in \cite[Theorem 1.5.4, p.~74]{GELFAND2005}, see also Section \ref{BACK}.
We shall use Study's determinant to establish an analogous result in the quaternionic setting.
\begin{proposition}\label{JACOBI} Let $A$ be an invertible quaternionic matrix. Then
 $$\sdet{(A^{-1})_{I,J}}=\sdet{A^{J,I}}/\sdet A.$$
\end{proposition}
\begin{proof}
If $I=\{i_1,\dots,i_p\}$ we denote
$I^\prime=I+n=\{i_1+n,\dots,i_p+n\}$; analogously
 $J^\prime=J+n$. The result follows from Lemma \ref{KRON} because
$$c\left((A^{-1})_{I,J}\right)=c(A^{-1})_{I\cup I^\prime,J\cup J^\prime}=\left(c(A)^{-1}\right)_{I\cup I^\prime,J\cup J^\prime}$$
and $$c(A^{J,I})=c(A)^{J\cup J^\prime, I\cup I^\prime}.$$ \end{proof}

\section{CHARACTERISTIC EQUATION}\label{CHARSECT} The problem we are proposing here is to find a characteristic map for the left eigenvalues of a given matrix $A$, that is to find a map $\mu_A\colon \HH \to \HH$ such that $\mu_A(\lambda)=0$ if and only if $\lambda$ is a left eigenvalue of $A$. Notice that the function $\sdet(A-\lambda\id)$ is real-valued, so it is not of interest from the point of view of the topological degree, nor it is solvable in any obvious way.

\subsection{Left eigenvalues}
Let $A$ be a  matrix with quaternionic coefficients.

\begin{definition} The quaternion $\lambda\in\HH$ is a {\em left eigenvalue} of $A$ if the matrix $A-\lambda \id$ is not invertible, or equivalently $\sdet(A-\lambda \id)=0$.
\end{definition}

Let  $\sigma_l(A)$ be the left spectrum, i.e. the set of left eigenvalues, of the matrix $A$.

\begin{proposition} The spectrum $\sigma_l(A)$ is compact.
\end{proposition}
\begin{proof}
The spectrum is a closed set because it is given by the equation  $\sdet(A-\lambda\id)=0$. It is bounded because $\lambda\in\sigma_l(A)$ if and only if there exists $v\in\HH^n$, $v\neq 0$, such that $Av=\lambda v$;
then
$$
  \vert\lambda\vert=\frac{\vert \lambda v\vert }{\vert v\vert}\leq\sup\limits_{\vert w \vert=1}\frac{\vert Aw\vert}{\vert w\vert} =\vert  A \vert.
$$
\end{proof}

\begin{proposition}\label{INVERSO} Let $B$ be an invertible matrix. Then $\lambda\in\sigma_l(B)$ if and only if $\lambda^{-1}\in \sigma_l(B^{-1})$.
\end{proposition}
\begin{proof}
If $Bv=\lambda v$ then $B^{-1}(\lambda v)=B^{-1}Bv=v=\lambda^{-1}(\lambda v).$
\end{proof}

\subsection{Background}\label{BACK}
When the matrix $A$ is hermitian, all left eigenvalues are real numbers so they coincide with the  right eigenvalues \cite{FARENICK}. Moreover it is possible to define a true determinant for hermitian matrices \cite{KYRCHEI,ZHANG1997},
which allows to construct a characteristic polynomial
$p(t)=\det(A-t\id)$ with real variable. % Este polinomio caracter'istico coincide con la nueva funci'on caracter'istica que propondremos en la Subsecci'on \ref{DefFC}.

For the general case,   $\sdet(A)$ equals, up to an exponent,  the determinant of $AA^*$.
On the other hand,    Zhang \cite{ZHANG1997} pointed out that if the quaternionic matrix is decomposed as $A= X+\J Y$, with $X,Y\in\M_{n\times n}(\complex)$, then its left eigenvalues  $\lambda=x+\J y$, with $x,y\in\C$, are the roots of the function
$\sigma\colon \C\times \C \to \R$ given by
\begin{equation}\label{FINAL}
\sigma(x,y)=\det
\begin{bmatrix}
X-x\id  & -\overline{Y}+\overline{y}\id \cr
Y-y\id  & \overline{X}-\overline{x}\id \cr
\end{bmatrix}.
\end{equation}
This is equivalent to the equation $\sdet(A-\lambda\id)=0$.

Another approach is due to Gelfand {\it et al.}  \cite{GELFAND2005}. To each matrix $A\in\M_{n\times n}(\quaternions),$ they associated $n^2$ functions, that we shall call {\em quasicharacteristic functions}, defined by
$$f_{ij}(\lambda)=\vert \lambda\id-A\vert _{ij}, \quad 1\leq i,j \leq n,$$
where  $\vert \cdot \vert _{ij}$ is the $(i,j)-$quasideterminant.
Let us denote by
$A^{i,j}$ the submatrix of order $(n-1)$ obtained by suppressing the row $i$ and the column $j$ in the matrix $A\in\M_{n\times n}(\HH)$. Then quasideterminants are defined inductively by the formula
$$\vert A\vert _{ij}=a_{ij}-\sum a_{iq}(\vert A^{i,j}\vert _{pq})^{-1}a_{pj},$$
where the sum is taken over the indices  $p,q\in\{1,\ldots,n\}$ with $p\neq i, q\neq j$, such that the quasideterminant of  lower order $\vert A^{i,j}\vert _{pq}$ is defined and it is non-null (see Proposition 1.5 of  \cite{GELFAND1992}).\label{QUASIDETERM}

When the matrix $A$ is invertible, the entries of the inverse matrix $A^{-1}$ are exactly $a^{ij}=\vert A\vert _{ji}^{-1}$. In the commutative case this gives the well known formula $a^{ij}=(-1)^{i+j}\det A^{j,i}/\det A$. For quaternionic matrices, the norm of the quasideterminant $\vert A\vert _{ij}$ of $A\in\M_{n\times n}(\HH)$ verifies
\begin{equation}\label{MINIJACOBI}
 \left\vert  \vert A\vert _{ij}\right\vert\cdot\sdet(A^{i,j}) =\sdet(A).
\end{equation}
This is a particular case of Jacobi identity for quaternions (Proposition \ref{JACOBI}).

\begin{remark} {\rm
From Equation (\ref{MINIJACOBI}) it follows that the roots of the quasicharacteristic functions are left eigenvalues. However none of those functions gives the complete spectrum, as shown in the next Example. Also notice that the definition of {\em noncommutative left eigenvalue} considered in \cite[subsection 8.2, p.~128]{GELFAND2005} does not correspond to the notion we are discussing here.
}\end{remark}
\begin{example}{\rm
Let  $A=\begin{bmatrix}
\I & 0 & 0 \cr
 \K & \J & 0 \cr
 -3\I & 2\K & \K\cr
 \end{bmatrix}$. Then $\sigma_l(A)=
\{ \I,\J,\K\}$. The quasi-characteristic functions are
\begin{eqnarray*}
f_{11}(\lambda)&=&\lambda-\I,\\
f_{12}(\lambda)&=&-(\lambda-\I)\K(\lambda-\J),\\
f_{13}(\lambda)&=&-(\lambda-\I)\left(3\I-2\K(\lambda-\J)^{-1}\K\right)^{-1}(\lambda-\K),\\
f_{22}(\lambda)&=&\lambda-\J,\\
f_{23}(\lambda)&=&-\frac{1}{2}(\lambda-\J)\K(\lambda-\K),\\
f_{32}(\lambda)&=&-2\K-(\lambda-\K)\K(\lambda-\J),\\
f_{33}(\lambda)&=&\lambda-\K,
\end{eqnarray*}
while $f_{21}(\lambda)$ and $f_{31}(\lambda)$ are not defined.
}\end{example}

\subsection{Characteristic map}\label{DefFC}
We now introduce the notion of a characteristic map whose roots are the left eigenvalues, thus generalizing the usual characteristic polynomial in the real and complex settings. As we shall see this notion fits naturally with the equation of order $2$ given by Wood \cite{WOOD}, as well as with the procedure proposed by  So \cite {SO2005} in order to compute the left eigenvalues of $3\times 3$ matrices.
\begin{definition}
The map $\mu: \HH \to\HH$ is a \emph{characteristic map} of the matrix $A\in\M_{n\times n}(\quaternions)$ if, up to a constant, its norm verifies $\vert \mu(\lambda)\vert =\sdet(A-\lambda\id)$ for all  $\lambda \in\HH$.
\end{definition}

\begin{example}
{\rm Let
$D=\diag (q_1,\ldots,q_n)$ be a diagonal matrix. Then
$\mu(\lambda)=(q_1-\lambda)\cdots(q_n-\lambda)$
is a characteristic map for $D$. Analogously for a triangular matrix.
}\end{example}

Let us start with the
 $2\times 2$ matrix $A=\begin{bmatrix}%begin of matrix
a&b\cr
c&d\cr
\end{bmatrix}%end of matrix
$. If  $A$ is a diagonal matrix, then $\sigma_l(A)$ reduces to the elements in the diagonal. Otherwise,  we can always suppose that $b\neq 0$ (see Remark \ref{PERMUTA}). Moreover,
$\sdet(A-\lambda \id)$  does not change  after elementary transformations (Corollary \ref{PropSdet2}), for instance
$$\sdet{\begin{bmatrix}%begin of matrix
a-\lambda&b\cr
c&d-\lambda\cr
\end{bmatrix}}%end of matrix
=
\sdet{\begin{bmatrix}%begin of matrix
0&b\cr
c-(d-\lambda)b^{-1}(a-\lambda)&d-\lambda\cr
\end{bmatrix}}.$$%end of matrix
Then, as pointed out by Wood,  computing the left spectrum is equivalent to finding the roots of a characteristic map like
\begin{equation}\label{CARMAP}
\mu(\lambda)=c-(d-\lambda)b^{-1}(a-\lambda).
\end{equation}
\begin{remark} {\rm Huang \cite{HUANG2000} proposed another map when $c\neq 0$, namely
$(\lambda-a)c^{-1}(\lambda-d)-b$.
This polynomial is obtained by adding  $(\lambda-a)c^{-1}$ by the second row  to the first row. This expression is equivalent to
$b-(a-\lambda)c^{-1}(d-\lambda)$,
which is the one given by Wood at the end of \cite{WOOD} (there is a misprint in the original paper).
}\end{remark}

The left spectrum is not invariant by similarity. However, we shall use the following fact:
\begin{proposition} If $P$ is an invertible {\em real} matrix then
$\sdet(A-\lambda\id)=\sdet(PAP^{-1}-\lambda\id)$.
Hence $A$ and $PAP^{-1}$ have the same characteristic maps.
\end{proposition}

\begin{remark}\label{PERMUTA} {\rm  Let $A$ be a matrix of order $n\geq 2$, let  $P_{\alpha\beta}$ be the real matrix obtained by interchanging the rows $\alpha$ and $\beta$ in the identity matrix $I_n$.  Left (resp. right) multiplication by the matrix $P_{\alpha\beta}$   switches two rows (resp. columns) of $A$.  Now let $i\neq j$ be two indices and let $\pi$ be any permutation of  $\{1,\dots,n\}$ sending $i$ to $1$ and $j$ to $n$. Then $\pi$ can be written as a composition of transpositions, so by taking the product $P$ of the corresponding matrices $P_{\alpha \beta}$ we obtain the matrix $PAP^{-1}$ where the initial entry $a_{ij}$ of $A$ has moved to the place $(1,n)$.
}\end{remark}

%%%%%%%%%%%%%%%%%%%%%%%%%%
\section{SPECTRUM OF MATRICES OF ORDER $2$}\label{ORDERTWO} In the next paragraphs we shall classify the different possible spectra of $2\times 2$ quaternionic matrices depending on the  rank of the differential $\mu_{*\lambda}$ of a characteristic map.

\subsection{Preliminaries}
The characteristic map $\mu\colon \quaternions \to \quaternions$ given in (\ref{CARMAP}) can be extended to a continuous (or even differentiable) map $\mu \colon S^4 \to S^4$ on the sphere $S^4=\quaternions\cup \{\infty\}$, because $\lim \vert \mu(\lambda) \vert = \infty$ when $\vert \lambda \vert\to\infty$.
A rigorous proof of the following result can be found  in  Eilenberg-Steenrod's book \cite[pp.~304--310]{EILSTEEN}:
\begin{proposition}\label{POWER} A polynomial map like $\mu$ and the power map $\lambda^2$ are homotopic, hence they have the same topological degree, which equals $2$.
\end{proposition}

From Lemma \ref{ReglasDerivacion} it follows that the differential of   $\mu$ is given by
\begin{equation}\label{DIFER}
\mu_{*\lambda}(X)=Xb^{-1}(a-\lambda)+(d-\lambda)b^{-1}X.
\end{equation}
\subsection{Classification of left spectra}\label{CLASS}
Now we are in a position to reformulate the following result from Huang and So \cite{HUANGSO2001}, see also \cite{MP1,ZHANG2011}.
Let
$A=\begin{bmatrix}%begin of matrix
a&b\cr
c&d\cr
\end{bmatrix}%end of matrix
$
be a quaternionic matrix with $b\neq 0$, and denote
$$a_0=-b^{-1}c, \quad a_1=b^{-1}(a-d), \quad \Delta=a_1^2-4a_0.$$

\begin{theorem}[\cite{HUANGSO2001}] The matrix $A$ has one, two or infinite left eigenvalues. The latter case is equivalent to the following conditions: $a_0,a_1$  are real numbers such that
$a_0\neq 0$ and $\Delta<0$.
\end{theorem}

\begin{remark}\label{ESFERA}{\rm We shall call {\em spherical} the infinite case, because  the  spectrum
$
\sigma_l(A)=\{(1/2)(a+d+bq)\colon q^2=\Delta\}
$
is diffeomorphic to the sphere $S^2\subset \quaternions_0=\langle \I ,\J ,\K\rangle$.}
\end{remark}

Let $\lambda\in\quaternions$ be an eigenvalue of $A$, that is
 $\mu(\lambda)=0$  for the map $\mu$ in
(\ref{CARMAP}). In the next Propositions we shall apply  Theorem \ref{RANGO} to the differential $\Sigma=\mu_{*\lambda}$ computed in (\ref{DIFER}). Accordingly to the notation of Sylvester equation in Section \ref{SYLVESTER}, we denote
$$P=(d-\lambda)b^{-1}, \quad Q=b^{-1}(a-\lambda).$$
First we study the two non-generic cases.
\begin{proposition}\label{CERO} If $\rank \mu_{*\lambda}=0$, then $a_0,a_1$ are real numbers and $\Delta=0$. Moreover $\lambda$ equals $(a+d)/2$ and this is the only left eigenvalue of the matrix.
\end{proposition}

\begin{proof}
We know from Theorem \ref{RANGO} that $P=t\in\reals$ and $Q=-t$, then $a_1=-2t$ and $2\lambda=a+d$. From $\mu(\lambda)=0$ it follows that $a_0=+t^2$, then $\Delta=0$. Now it is easy to check (using for instance Theorem 2.3 in \cite{HUANGSO2001}) that $\lambda=a+tb$ is the only left eigenvalue of $A$.
\end{proof}

\begin{lemma}\label{SOLO}  Let $A,B$ be two similar quaternions that do not commute. Then the equation
$\lambda^2-(A+B)\lambda+AB=0$ has the unique solution $\lambda=B$. \end{lemma}

\begin{proof}
If $\lambda\neq B$ is a solution,  it follows from  $(\lambda - B)\lambda=A(\lambda-B)$  that $\lambda$ and $A$ are similar, then $\Re(\lambda)=\Re(A)=\Re(B)$ and $\vert \lambda \vert= \vert A\vert= \vert B \vert$.
By substituting in the equation we see that the  real parts and norms dissapear, so we can suppose that $A,B,\lambda$ are pure imaginary quaternions with norm $1$. Hence $\lambda^2=-1=B^2$ so the equation reduces to $(A+B)\lambda=(A+B)B$, which implies $\lambda=B$, a contradiction.

Alternatively, the uniqueness of $\lambda$ can be proved by using Theorem 2.3, case 4.1, of the solution of quadratic equations in \cite{HUANGSO2002}.
\end{proof}

\begin{proposition}\label{DOS}
If $\rank \mu_{*\lambda}=2$   two things may happen:
\begin{enumerate}
\item either the spectrum is spherical and all the eigenvalues have rank $2$;
\item
or the matrix has just one eigenvalue.\end{enumerate}
\end{proposition}

\begin{proof}
By using the diffeomorphism $a+b\sigma_l(A^\prime)=\sigma_l(A)$ we can substitute $A$ by  the so-called ``companion matrix"
$A^\prime=\begin{bmatrix}%begin of matrix
0&1\cr
-a_0&-a_1\cr
\end{bmatrix}%end of matrix
$.
Since the rank is $2$, we have from Theorem \ref{RANGO} that
$P=t+\alpha$ and $Q=-t+\beta$ with $\alpha,\beta\in \quaternions_0=\langle \I ,\J ,\K\rangle$, $\vert \alpha \vert =\vert \beta \vert \neq 0$. Then $a_1=-2t+\beta- \alpha$.
The first possibility is that $\beta=\alpha$, then $a_1=-2t$. It follows from $\mu(\lambda)=0$ that $a_0=t^2+\vert \beta \vert^2 \neq 0$ and $\Delta=-4\vert \beta \vert^2 <0$. Then we have the spherical case. In particular $\lambda=(-a_1-2\beta)/2$. The other eigenvalues have the form $(-a_1+q)/2$ with $q^2=-4\vert\beta\vert^2$, then the differential of $\mu$ verifies $P=
t- q^{-1}/2$ and $Q=-t-q/2$, and so they have rank $2$ too.
The second possibility is that $\beta\neq \alpha$. Then
$a_1=-2t+\beta-\alpha$, $a_0=(t+\alpha)(t-\beta)$, and Lemma \ref{SOLO} shows that the only eigenvalue is $\lambda=t-\beta$.
\end{proof}

Now we consider the generic case.

\begin{proposition} If $\rank \mu_{*\lambda}=4$  then the matrix has two different eigenvalues.
\end{proposition}

\begin{proof}
Since the differential has maximal rank at the eigenvalue $\lambda$, the matrix $A$ cannot correspond to Propositions \ref{CERO} or \ref{DOS}, hence all its eigenvalues are of rank $4$.  Then by the inverse function theorem the fiber $\mu^{-1}(0)$ is discrete (in fact compact) and its cardinal equals (Theorem \ref{REGULAR}) the degree of the map $\mu$, which is $2$ by Proposition \ref{POWER}. Notice that  the Jacobian is nonnegative by formula (\ref{DET}).
\end{proof}

\begin{remark} {\rm In \cite{JANOVSKAOPFER2010}, Janovsk\'a and Opfer show that for quaternionic polynomials there are several types of zeros
accordingly to the rank of some real $4\times 4$ matrix, but their procedure does not seem to have an immediate geometrical meaning.}
\end{remark}
%%%%%%%%%%%%%%%
\section{CHARACTERISTIC MAPS OF $3\times 3$ MATRICES}\label{CHARMAP3x3}
The only known results about the left spectrum of $3\times3$ matrices are due to So \cite{SO2005}, who did a case by case study, depending on some relationships among the entries of the matrix. He showed that when $n=3$ left eigenvalues  could
be found by solving quaternionic polynomials of degree not greater than $3$. In general there is not any known method for solving the resulting equations.

In the following paragraphs we shall develop a method for matrices of order $3$ which is analogous to that of  Section \ref{ORDERTWO}, that is we shall find a map $\mu_A$ such that $\vert \mu_A(\lambda) \vert  = \sdet(A-\lambda\id)$. This time, however, the characteristic map $\mu_A$ will be in most cases a rational function instead of a polynomial (the latter occurs when the matrix $A$ has some null entry outside the diagonal).

Let us consider the quaternionic matrix
$A=\begin{bmatrix}
a & b & c \cr
f & g & h \cr
p & q & r \cr
\end{bmatrix}\in \M_{3\times 3}(\quaternions)$.

%%%%%%%%%%%%  CASO   C=0  %%%%%%
\subsection{Polynomial case}\label{SSC0}
We start  studying the simplest situation, when there exists some zero entry outside the diagonal.  

First, suppose that the matrix has the zero entry $c=0$, that is
$$\sdet(A-\lambda\id)=
\sdet\begin{bmatrix}
a-\lambda & b & 0 \cr
f & g-\lambda & h \cr
p & q & r-\lambda\cr
\end{bmatrix}.$$
There are three possibilities:
\begin{enumerate}
\item if $b,h=0$, we have a triangular matrix, so we can take
\begin{equation}\label{BH0}
\mu(\lambda)=(r-\lambda)(g-\lambda)(a-\lambda);
\end{equation}
\item if $b= 0$ but $h\neq 0$, then Proposition \ref{sdet por cajas} allows us to reduce to the  $2\times 2$ case  and we obtain
\begin{equation}\label{B0}
\mu(\lambda)=\left(q-(r-\lambda)h^{-1}(g-\lambda)\right)(a-\lambda);
\end{equation}
\item finally, if $b\neq 0$, we can proceed as follows. We create a zero in the first row by substracting to the first row  $C_1$ the multiple $C_2b^{-1}(a-\lambda)$ of the second column:
$$\sdet(A-\lambda\id)=\sdet\begin{bmatrix}
0 & b & 0 \cr
f-(g-\lambda)b^{-1}(a-\lambda) & g-\lambda & h \cr
p-qb^{-1}(a-\lambda) & q & r-\lambda\cr
\end{bmatrix},$$
then we permute the two last columns in order to reduce to the $2\times 2$ case. In this way we can take as a characteristic map the polynomial of degree $3$
\begin{equation}\label{POLINOMIO}
\mu(\lambda)=p-qb^{-1}(a-\lambda)-(r-\lambda)h^{-1}\left(f-(g-\lambda)b^{-1}(a-\lambda)\right).
\end{equation}
\end{enumerate}

\begin{theorem}\label{polinomio} If the matrix $A\in\M_{3\times 3}(\HH)$ has some zero entry outside the diagonal, then $A$ admits a polynomial characteristic map.
\end{theorem}
\begin{proof} Let the entry be $a_{ij}=0$,  with $i\neq j$. Then accordingly to Remark \ref{PERMUTA} there is a real invertible matrix $P$ such that the transformation $PAP^{-1}$ does not change the characteristic maps and gives a matrix with $a_{13}=0$.
\end{proof}

%%%%%%%%%%%%   CASO C no 0  %%%%%
\subsection{Rational case}\label{CNO0}
In the more general situation, when $c\neq 0$, we can compute the Study's determinant of the matrix $A$ by creating zeroes in the first row. Then
$$
\sdet(A)=\sdet\begin{bmatrix}
0&0&c \cr
f-hc^{-1}a & g-hc^{-1}b&h \cr
p-rc^{-1}a &q-rc^{-1}b&r \cr
\end{bmatrix}.
$$

From Lemma \ref{sdet por cajas} and the results for $2\times 2$ matrices it follows:
\begin{proposition}\label{CASOCno0}
If $c\neq 0$, then $\sdet(A)$ is given:
\begin{enumerate}
\item
when $g-hc^{-1}b \neq 0$, by
$$
\vert c\vert \cdot \vert g-hc^{-1}b\vert  \cdot \vert p-rc^{-1}a-(q-rc^{-1}b)(g-hc^{-1}b)^{-1}(f-hc^{-1}a)\vert ;
$$
\item
otherwise, by
$$
\vert c\vert \cdot \vert q-rc^{-1}b\vert  \cdot \vert f-hc^{-1}a\vert .
$$
\end{enumerate}
\end{proposition}

\begin{definition}\label{poloDef}
We shall call \emph{pole} of the matrix   $A\in\mathcal{M}_{3\times 3}(\HH)$ the point $\pi_A=g-hc^{-1}b$.
\end{definition}

Notice that $\pi_A$ is the quasideterminant $\vert A^{3,1}\vert_{21}$ (see page \pageref{QUASIDETERM}).

By applying Proposition~\ref{CASOCno0}  to the matrix $A-\lambda \id $ we obtain the following characteristic map for  $A$ (we omit the term $\vert c \vert$).

\begin{proposition}\label{CHARTRES}
Let  $A$ be a matrix of order  $3\times 3$ such that $c\neq 0$. A characteristic map can be defined as follows:
\begin{enumerate}
\item
if $\pi_A=g-hc^{-1}b$ is the pole of $A$, then
$$\mu(\pi_A)=\left(q-(r-\pi_A)c^{-1}b\right)\left(f-hc^{-1}(a-\pi_A)\right);$$
\item
for $\lambda\neq\pi_A$ we define
\begin{eqnarray}\label{MURACIONAL}
\mu(\lambda)&=&
(\pi_A-\lambda)
\left[p-(r-\lambda)c^{-1}(a-\lambda)\right.-\\
&&\quad\left.\left(q-(r-\lambda)c^{-1}b\right)
	 (\pi_A-\lambda)^{-1}
	\left(f-hc^{-1}(a-\lambda)\right)\right].\nonumber
\end{eqnarray}
\end{enumerate}
\end{proposition}

\begin{remark} {\rm The map in (\ref{MURACIONAL})  is exactly the same formula given by So in \cite[p.~563]{SO2005}, even if our method is completely different. This is why we chose to compute determinants starting from the right bottom corner. }
\end{remark}

\subsection{Continuity at the pole}
Up to now we have defined maps which verify
$\vert \mu(\lambda)\vert= \sdet(A-\lambda\id)$ in norm. Since  $\sdet$ is a continuous map  we have
$\lim\limits_{\lambda\rightarrow \pi_A}\vert \mu(\lambda)\vert =\vert \mu(\pi_A)\vert $. However, the following example shows that $\mu$ may not be continuous at the pole $\pi_A$.

\begin{example}\label{PoloNoAutovalor}
{\rm
Let $A=\begin{bmatrix}
0 & \I & 1 \cr
3\I-\K & 0 & 1 \cr
\K & -1+\J+\K & 0\cr
\end{bmatrix}$. The pole $\pi_A= -\I$ is not a left eigenvalue; in fact
$$\mu(\pi_A)=(-1+\J+\K+1)(3\I-\K-\I)=1-\I+2\J-2\K.$$
We observe that the directional limits
$$\lim\limits_{\varepsilon\to 0}\mu(-\I+\varepsilon q)= -q(\J+\K)q^{-1}(2\I-\K)$$
depend on  $q\in\HH$, hence $\lim\limits_{\lambda\to\pi_A} \mu(\lambda)$ does not exist.
}\end{example}

\begin{theorem}\label{CONTINUIDAD} The characteristic rational map $\mu_A$ is continuous if and only if the pole $\pi_A$ is a left eigenvalue of $A$.
\end{theorem}
\begin{proof}
Assume that $\pi_A$ is a left eigenvalue. Let $(q_n)_n$ be a sequence converging to $\pi_A$. Then
$\vert \mu(q_n)\vert=\sdet(A-q_n\id)$ converges to $ \sdet(A-\pi_A\id)=0$,
that is $\mu(q_n)\to 0=\mu(\pi_A)$.

Now we shall prove the converse. The map $\mu$ defined in Proposition \ref{CHARTRES} is of the form
\begin{equation}\label{MU}
\mu(\lambda)=(\pi_A-\lambda)\left[p(\lambda)-q(\lambda)(\pi_A-\lambda)^{-1}f(\lambda)\right], \quad \lambda\neq \pi_A,
\end{equation}
while $\mu(\pi_A)=q(\pi_A)f(\pi_A).$ Let as assume that $\lim\limits_{\lambda\to\pi_A}{\mu(\lambda)}$ exists and equals $\mu(\pi_A)$. We must check that $\mu(\pi_A)=0$. If $f(\pi_A)=0$ we have finished. Otherwise we deduce from (\ref{MU}) that
\begin{equation}\label{MU2}
\lim\limits_{\lambda\to\pi_A}{(\lambda-\pi_A)q(\lambda)(\lambda-\pi_A)^{-1}}=-q(\pi_A).
\end{equation}
 From Lemma \ref{MILEMA} it follows  that the limit on the left side equals $q(\pi_A)$, hence $q(\pi_A)=0$, which ends the proof.
 \end{proof}

\begin{lemma}\label{MILEMA} Let $Q=Q(\lambda)$ be a continuous map and suppose that there exists the limit
$l_0=\lim\limits_{\lambda\to 0}{\lambda Q(\lambda)\lambda^{-1}}.$ Then $l_0=Q(0)$.
\end{lemma}

\begin{proof}
Take a sequence of real numbers $(\varepsilon_n)_n\to 0$. Then
$$l_0=\lim \varepsilon_n Q(\varepsilon_n)\varepsilon_n^{-1}=\lim Q(\varepsilon_n)=Q(0). $$
\end{proof}

Notice that the differentiability of $\mu$ at the pole $\pi_A$ is not ensured.

It is an open question whether it is always possible or not to find a polynomial, or at least a continuous characteristic map for a given matrix $A$.

\subsection{Extension to the infinite point}
Each of the characteristic maps we have introduced up to now can be extended to the sphere $S^4= \HH\cup\{\infty\}$.
\begin{proposition}
The polynomial maps $\mu$ defined in Subsection \ref{SSC0} verify that $\lim\limits_{\vert \lambda\vert \to\infty}\vert \mu(\lambda)\vert =\infty.$
\end{proposition}
\begin{proof}
When $b=0$, the result follows from formulae (\ref{BH0}) and (\ref{B0});
when $b\neq 0$, the expression of $\mu$ is that of formula (\ref{POLINOMIO}),
so
$$
\frac{\vert \mu(\lambda)\vert }{\vert \lambda\vert ^2}\geq
\frac{\vert (r-\lambda)h^{-1}(g-\lambda)b^{-1}(a-\lambda)\vert }{\vert \lambda\vert ^2}
-
\frac{\vert -p+qb^{-1}(a-\lambda)+(r-\lambda)h^{-1}f\vert }{\vert \lambda\vert ^2}.
$$
\end{proof}

\begin{proposition}\label{ADMISIBLE}
The rational map $\mu$ defined in the $c\neq 0$ case  by formula (\ref{MURACIONAL}) can be extended to $S^4=\HH \cup \{\infty\}$ (maybe with a discontinuity at the pole $\pi_A$).
\end{proposition}
\begin{proof}
We have
$$\mu(\lambda) = (\pi_A-\lambda)p_2(\lambda)-(\pi_A-\lambda)q_1(\lambda)(\pi_A-\lambda)^{-1}f_1(\lambda)
$$
with polynomials
\begin{eqnarray}\label{FORMULAS}
p_2(\lambda)&=&p-(r-\lambda)c^{-1}(a-\lambda), \\\nonumber
q_1(\lambda)&=&q-(r-\lambda)c^{-1}b, \\\nonumber
f_1(\lambda)&=&f-hc^{-1}(a-\lambda).
\end{eqnarray}
Then
$$
\vert \mu(\lambda)\vert  \geq \vert (\pi_A-\lambda)p_2(\lambda)\vert -\vert q_1(\lambda)f_1(\lambda)\vert    $$
and
$$\lim\limits_{\vert \lambda \vert \to \infty}{\frac{\vert\mu(\lambda)\vert}{\vert \lambda \vert^3}}\geq\lim\limits_{\vert \lambda \vert \to \infty}{\frac{\vert(\pi_A-\lambda)p_2(\lambda)\vert}{\vert \lambda \vert^3}}=\lim\limits_{\vert \lambda \vert \to \infty}{\vert c^{-1}\vert\frac{\vert r-\lambda \vert}{\vert \lambda \vert}\frac{\vert a-\lambda \vert}{\vert \lambda \vert}}=\vert c^{-1}\vert.$$
\end{proof}
\subsection{Discontinuous case}
Let us now assume that the characteristic map $\mu_A$ defined in Section \ref{CHARTRES} is not continuous at the pole $\pi_A$, or equivalently that $\pi_A$ is not a left eigenvalue of $A$ (Theorem \ref{CONTINUIDAD}). Then the matrix $B=A-\pi_A\id$ is invertible and its pole is $\pi_B=0$. Moreover $\sigma_l(A)=\sigma_l(B)+\pi_A$. On the other hand, from Proposition \ref{INVERSO} we know that the spectra of $B$ and $B^{-1}$ are diffeomorphic, because
$\sigma_l(B^{-1})=\sigma_l(B)^{-1}$.

\begin{theorem}\label{POLINV} Let $A$ be a matrix such that the pole $\pi_A$ is not a left eigenvalue. Let $B=A-\pi_A\id$. Then the matrix $B^{-1}$ has a polynomial characteristic map.
\end{theorem}

\begin{proof} Accordingly to formula (\ref{MINIJACOBI}) the norm of the entry $(1,3)$ of the matrix $B^{-1}$ equals
$$\sdet(B^{3,1})/\sdet(B)=\vert\pi_B \vert/\sdet(B)=0,$$ then Theorem \ref{polinomio} applies.
\end{proof}

Here is an alternative proof of Theorem \ref{POLINV}. Let $\mu_B(\lambda)=-\lambda R(\lambda)$, with $$R(\lambda)=p(\lambda)+q(\lambda)\lambda^{-1}f(\lambda), \quad \lambda\neq 0,$$  be the characteristic map
  given in (\ref{CHARTRES}) (we assume $\pi_B=0$.) Then it is immediate that $\lambda R(\lambda^{-1})\lambda$ is a polynomial in $\lambda$ (of degree $3$ and independent term $-c^{-1}$) and we only have  to apply the following result.

\begin{proposition}\label{CARINVERSO} Let $\mu_B$ be a characteristic map of the invertible matrix $B=A-\pi_A\id$, with a discontinuity at the pole $\pi_B=0$. Then
$$\mu_{B^{-1}}(\lambda)=\sdet(B)^{-1}\lambda^2\mu_B(\lambda^{-1})\lambda, \quad \lambda\neq 0,$$ is a characteristic map for $B^{-1}$.
\end{proposition}

\begin{proof}
From Proposition \ref{propiedadessdet} we have
$$\sdet(\lambda^{-1}\id)\cdot\sdet(B^{-1}-\lambda\id)\cdot\sdet(B)=\sdet(\lambda^{-1}\id-B)$$
then
$$\vert\lambda^{-3}\vert\cdot \sdet(B^{-1}-\lambda\id)\cdot\sdet(B)=\vert \mu_B(\lambda^{-1})\vert.$$
\end{proof}

\begin{remark}\label{TRICK}{\rm The idea that a rational map like $R(\lambda)$ can be converted into a polynomial $\lambda^{-1}R(\lambda)\lambda^{-1}$ with variable $\lambda^{-1}$ is due to So (see \cite[Lemma 3.5, p.~558]{SO2005}).
}\end{remark}

%%%%%%%%%%
\section{TOPOLOGICAL STUDY OF THE  $3\times3$ CASE}\label{3x3Topologico}
%Para matrices de orden $3$, al hacer la diferencial de la funci'on caracter'istica \ref{CHARTRES}, obtenemos de nuevo una ecuaci'on lineal que podemos discutir con el m'etodo visto en \ref{ResEcLineales}. Por ahora no podemos establecer una clasificaci'on completa pero el m'etodo nos permite hallar el rango de la diferencial $\mu_{*\lambda}$ para cada uno de los autovalores por la izquierda de una matriz cuaterni'onica de orden tres arbitaria.

We shall consider separately the polynomial, rational continuous and discontinuous cases considered in Section \ref{CHARMAP3x3}.

\subsection{Polynomial case}
Let us start with matrices having some null entry outside the diagonal (as we have seen this case can be reduced to the case  $c=0$). We know  that  the characteristic map is a polynomial of degree  $3$ which can be extended in a continuous way to the sphere  $S^4$. Then  since there is a unique term of higher degree $3$, the map $\mu_A$ is homotopic to $\lambda^3$, so it has topological degree $3$.
%\subsection{Linearization}
\begin{proposition} Let $\lambda$ be a left eigenvalue of the matrix
$A$ with  $c=0$. Then the differential of the polynomial characteristic map $\mu_A$ in Subsection \ref{SSC0} is given by
\begin{enumerate}
\item if $b,h=0$ then $$\mu_{*\lambda}(X) =-X(g-\lambda)(a-\lambda)-(r-\lambda)(g-\lambda)X-(r-\lambda)X(a-\lambda);$$
\item if $b=0, h\neq 0$ then
\begin{eqnarray*}
&&\mu_{*\lambda}(X) =\\
&&\quad Xh^{-1}(g-\lambda)(a-\lambda)-\left(q-(r-\lambda)h^{-1}(g-\lambda)\right)X+(r-\lambda)h^{-1}X(a-\lambda);
\end{eqnarray*}
\item otherwise,
\begin{eqnarray*}
\mu_{*\lambda}(X) &=&\left(qb^{-1}-(r-\lambda)h^{-1}(g-\lambda)b^{-1}\right)X \\
&&\quad+Xh^{-1}\left(f-(g-\lambda)b^{-1}(a-\lambda)\right)-(r-\lambda)h^{-1}Xb^{-1}(a-\lambda).
\end{eqnarray*}
\end{enumerate}
\end{proposition}

The proof is a direct application of the derivation rules given in Lemma \ref{ReglasDerivacion}.

The expressions obtained are of the form  $PX+XQ+RXS=0$, whose rank can be computed with the method given in Section  \ref{ResEcLineales}.

\begin{example} {\rm Let $A=\begin{bmatrix}
a & 0 & 0 \cr
f & g & 0 \cr
p & q & r\cr
\end{bmatrix}$ be a triangular matrix.
The differential $\mu_{*\lambda}(X)$ of the characteristic map
$$\mu(\lambda)=(r-\lambda)(g-\lambda)(a-\lambda),$$
at the eigenvalues $\lambda=a,  g, r$ is given, respectively, by
%$$\mu_{^*\lambda}(X)=-X(g-\lambda)(a-\lambda)-(r-\lambda)(g-\lambda)X-(r-\lambda)X(a-\lambda).$$
$(a-r)(g-a)X$, $(g-r)X(a-g)$ and
$(r-g)(a-r)X$. Hence, unlike the case $n=2$,  the rank  depends on the multiplicity of each eigenvalue, and can be either $0$ or $4$.
}\end{example}

\subsection{Rational case} When none of the entries outside the diagonal is zero, the characteristic map is a rational function, with a distinguished point $\pi_A$.

Let us first suppose that the pole $\pi_A$ is a left eigenvalue. We know from (\ref{MURACIONAL}) that $\mu_A$ is a continuous map  on $S^4$ of the form
$$(\pi_A-\lambda)\left[p(\lambda)-q(\lambda)(\pi_A-\lambda)^{-1} f(\lambda)\right].$$  By examining formulae (\ref{FORMULAS}) it is clear that $p(\lambda)$ is homotopic to $-\lambda^2$ by the homotopy
  $$t p-(tr-\lambda)(1-t+tc^{-1})(ta-\lambda), \quad t\in [0,1].$$
 Analogously $q(\lambda)\sim \lambda$ and $f(\lambda)\sim \lambda$, so $\mu_A(\lambda)$ is homotopic to $$(\pi_A-\lambda)\left[-\lambda^2-\lambda(\pi_A-\lambda)^{-1}\lambda\right]$$ (notice that this map is continuous at $\lambda=0$), which in turn is homotopic to $\lambda^3-\lambda^2$ by the homotopy
$$(t\pi_A-\lambda)\left[-\lambda^2-\lambda(t\pi_A-\lambda)^{-1}\lambda\right].$$ Finally the homotopy $\lambda^2(\lambda-t)$ shows that the map $\mu_A$ is homotopic to $\lambda^3$. All these homotopies can be extended to the infinity.

Hence we have proved

 \begin{proposition} When the rational characteristic map $\mu_A$ is continuous  it has topological degree $3$.
 \end{proposition}

On the other hand, suppose that $\pi_A$ is not a left eigenvalue. Then the polynomial case applies to $(A-\pi_A\id)^{-1}$ by Theorem \ref{POLINV}. So we do not have to use the local theory of degree, whose main difficulty is the need of considering homotopies which are {\em admissible} with respect to  the domain $\Omega$ of definition, see \cite[p.~28]{AMBROSETTI}.
%
%\subsection{}
%Para matrices $4\times 4$ puede construirse de modo an'alogo una funci'on caracter'istica racional.

\begin{corollary} Any $3\times 3$ quaternionic matrix has at least one left eigenvalue.
\end{corollary}

\begin{proof} In all cases the eigenvalues (or its inverses) can be computed as the roots of a continuous map $\mu$ of degree $3$; then, $\mu^{-1}(0)$ is not void (see Section \ref{TOPOL}).
\end{proof}

\subsection{Final remarks}
 In order to simplify the computation of the rank, by taking $B=A-\pi_A\id$ we  can  always assume    that the pole is $\pi_B=0$.

\begin{proposition}
For a  $3\times 3$ matrix with $c\neq 0$ the differential of the characteristic map  $\mu$ given in formula (\ref{MURACIONAL}) at the point $\lambda\neq \pi_B=0$ is
\begin{eqnarray*}
&&\mu_{*\lambda}(X)= \\
 &&X\left[-p+(r-\lambda)c^{-1}(a-\lambda)+(q-(r-\lambda)c^{-1}b)(-\lambda)^{-1}(f-hc^{-1}(a-\lambda))\right]  \\
&&\quad +(-\lambda)Xc^{-1}(a-\lambda)-(-\lambda)Xc^{-1}b(-\lambda)^{-1}(f-hc^{-1}(a-\lambda)) \\
&&\quad-(-\lambda)(q-(r-\lambda)c^{-1}b)(-\lambda)^{-1}X(-\lambda)^{-1}(f-hc^{-1}(a-\lambda)) \\
&&\quad+\left[(-\lambda)(r-\lambda)c^{-1}-(-\lambda)(q-(r-\lambda)c^{-1}b)(-\lambda)^{-1}hc^{-1}\right]X.
\end{eqnarray*}
\end{proposition}
When $\mu_A$ is continuous but not differentiable at the pole $\pi_A$ the rank at $\pi_A$ could be computed by taking a different characteristic map $\mu_{PAP^{-1}}$.
In particular, by moving around the entries outside the diagonal of a given matrix (see Remark~\ref{PERMUTA} of Section~\ref{CNO0}) one can obtain up to six different characteristic maps, each one with a different pole. 

It is an open question whether there exists a matrix $A$ verifying that all the poles of the matrices $PAP^{-1}$, with $P$ real, are eigenvalues. If such an example does not exist then the non-polynomial case would not be necessary.

\section{EXAMPLES}\label{EXAMPLES}
Here are some miscellaneous examples.

\begin{example} Discontinuous map.
{\rm
Let
$A$ be the matrix given in Example \ref{PoloNoAutovalor}. Then
$$B=A-\pi_A\id=
\begin{bmatrix}\I & \I & 1\cr
3\I-\K & \I & 1\cr
\K & -1+\J+\K & \I\cr
\end{bmatrix}$$ is an invertible matrix with pole $\pi_B=0$.
By computing the quasideterminants we obtain the inverse matrix
$$B^{-1}=\frac{1}{10}\begin{bmatrix}
4\I-2\K & -4\I+2\K & 0 \cr
-1-3\I+8\J-6\K & 1+3\I-3\J+\K & -5\J-5\K \cr
11+\I-8\J-8\K & -1-\I+3\J+3\K & -5\J+5\K \cr
\end{bmatrix}
.$$
Its polynomial characteristic map is
\begin{eqnarray}\label{CALCULO}
&&\quad\mu_{B^{-1}}(\lambda)=\\ \nonumber
&&10-\lambda\I-2\I\lambda-\frac{1}{10}\I\lambda(2\I-\K)\lambda-\frac{1}{10}\lambda(1+\J+2\K)\lambda-\frac{1}{100}\lambda(\J+\K)\lambda(2\I-\K)\lambda.
\end{eqnarray}
On the other hand, the rational characteristic map for $B$ is
%\begin{eqnarray*}
%&&
$$\mu_B(\lambda) = -\lambda R(\lambda)=%\\&&
 -\lambda\left[ 1+\K+\I\lambda+\lambda\I-\lambda^2+(\J+\K+\lambda\I)\lambda^{-1}(2\I-\K+\lambda)\right]
$$
%\end{eqnarray*}
and one can check that $\lambda R(\lambda^{-1})\lambda$ equals (up to a constant) the map (\ref{CALCULO}).
}\end{example}

\begin{example} An eigenvalue of rank $3$.
{\rm
Let
$$A=\begin{bmatrix}\J & 1 & 0 \cr
2\I & -\K & 1 \cr
2-\I-2\J & -1-\J+\K & -\I-\K\cr
\end{bmatrix}.$$ The characteristic map is
$$\mu(\lambda)=2-\I-2\J+(1+\J-\K)(\J-\lambda)+(\I+\K+\lambda)\left(2\I+(\K+\lambda)(\J-\lambda)\right).$$ For the left eigenvalue  $\lambda=0$ the differential is $$\mu_{*0}(X)=\K X+X\I+(\I+\K)X\J,$$  whose real associated matrix
$M=\begin{bmatrix} 0 & -2 & 0 & 0 \cr
 0 & 0 & -2 & 0 \cr
 0 & 0 & 0 & 0 \cr
 2 & 0 & -2 & 0\cr
 \end{bmatrix}$
 has rank $3$.
}\end{example}

\begin{example} A matrix which can not be reduced to the polynomial case $c=0$.
{\rm We shall exhibit a matrix $A\in \M_{3\times 3}(\quaternions)$ such that for any real invertible matrix $P$  all entries in the matrix $PAP^{-1}$ outside the diagonal are not null. Let $A=T+\I X + \J Y + \K Z$, with $T,X,Y,Z\in \M_{3\times 3}(\reals)$. Then
$$PAP^{-1}=PTP^{-1}+\I PXP^{-1} + \J PYP^{-1} + \K PZP^{-1},$$ which means that the matrix $PAP^{-1}$ has a null entry if and only if the same happens for the real matrices $PTP^{-1}$, $PXP^{-1}$, $PYP^{-1}$ and  $PZP^{-1}$. Moreover, from Remark \ref{PERMUTA} we can suppose that the null entry is at the place $(1,3)$.
Now, recall  that $A$  is the matrix associated to a linear map $\HH^3 \to \HH^3$ with respect to the canonical basis $e_1,e_2,e_3$ and that the real matrix $P=(p_{ij})$ represents the change to another basis $v_1,v_2,v_3$, that is $v_j=\sum_i{p^{ij}e_i}$ where $P^{-1}=(p^{ij})$.  Assume that $T,X$ are the matrices associated to two rotations $\R^3 \to \R^3$ with the same rotation axis ${\cal L}$ and rotation angles $+\pi/2$ and $-\pi/2$ respectively. Then the nullity of the entry (1,3) means that $Tv_3, Xv_3 \in \langle v_2,v_3\rangle$, which implies that either $v_3$ is in the direction of the axis, i.e. ${\cal L}=\langle v_3\rangle$, or it is orthogonal to the axis, in which case $\langle v_2,v_3\rangle ={\cal L}^\perp$ (otherwise it is impossible that $v_3$, $Tv_3$ and $Xv_3$ lie in the same plane). Now, suppose that $Y,Z$ are two other rotations with axis ${\cal L}^\prime$ and rotation angles $\pm \pi/2$. If ${\cal L}$ and  ${\cal L}^\prime$ are different and not perpendicular, then it is impossible that $Yv_3, Zv_3 \in \langle v_2,v_3\rangle$.
}
\end{example}

\begin{example}\label{CONTI} A continuous rational characteristic map.
{\rm
The pole $\pi_A= 1+\J$ is an eigenvalue of the matrix
$$A=\begin{bmatrix}
0 & -\J & \I \cr
-1+\J & \J & \K \cr
p & q & r\cr
\end{bmatrix},$$
with $p, q, r\in \HH$ arbitrary, $p,q\neq 0$. In fact,
$$\mu(\pi_A)= \left(q-(r-1-\J)\K\right)\left(-1+\J+\J(-1-\J)\right)=0.$$
}\end{example}

\begin{example}  Generic polynomial case.
{\rm Let
$$A=\begin{bmatrix}\K & 0 & 0 \cr
3\I-\J & -\I & \I \cr
1-2\K & \J & -\J\cr
\end{bmatrix}.$$
The characteristic map is $\mu(\lambda)=(-1-\K+\lambda\I)\lambda(\K-\lambda)$
hence $\sigma_l(A)=\{\K,0,-\I-\J\}$.
The differential of $\mu$ at each eigenvalue is
\begin{eqnarray*}
\mu_{*\K}(X) &=& (-1-\I+\K)X, \\
\mu_{*0}(X) &=& -(1+\K)X\K,\\
\mu_{*(-\I-\J)}(X) &=& X(1+2\I+\K).
\end{eqnarray*}
Then the matrix $A$ has  three different eigenvalues, all of them with maximal rank.
}\end{example}

\begin{example}
Two eigenvalues, one of null rank, the other one of maximal rank.
{\rm Let
$$A=\begin{bmatrix}-\I-\J & 0 & 0 \cr
\K & -\I & \I \cr
1-\I & \J & -\J \cr
\end{bmatrix}.$$
This time $\mu(\lambda)=(1+\K-\lambda\I)\lambda(\I+\J+\lambda)$ so $\sigma_l(A)=\{0,-\I-\J\}$.
We have
\begin{eqnarray*}
\mu_{*0}(X)&=&(1+\K)X(\I+\J), \\
\mu_{*(-\I-\J)}(X) &=& 0.
\end{eqnarray*}
}\end{example}

%%%%%%%%%%%

{\small
\begin{thebibliography}{99}
\bibitem{AMBROSETTI} Ambrosetti, A.; Malchiodi, A. {\em Nonlinear analysis and semilinear elliptic problems}. Cambridge Studies in Advanced Mathematics, 104. Cambridge University Press, Cambridge, 2007.

\bibitem{ASLAKSEN1996}  Aslaksen, H.
Quaternionic determinants.
{\em Math. Intell.} {\bf 18} No.3, 57-65 (1996).

\bibitem{COHENLEO2000}  Cohen, N.; De Leo, S.
The quaternionic determinant.
{\em Electron. J. Linear Algebra} {\bf 7} 100-111 (2000).

\bibitem{DEIMLING}  Deimling, K.
\newblock {\it Nonlinear functional analysis}. Springer-Verlag, 1985.

\bibitem{DIEUDONNE} Dieudonn\'e, J.
\newblock {\it A history of algebraic and differential topology 1900--1960}. Modern Birkh\"auser Classics, Boston, 2009.

\bibitem{EILSTEEN} Eilenberg, S.; Steenrod, N.
\newblock {\it Foundations of algebraic topology}.
Princeton Mathematical Series, No.15, Princeton: University Press, XIV, 1952.

\bibitem{FARENICK} Farenick, D. R.; Pidkowich, B. A. F.
The spectral theorem in quaternions.
 \textit{Linear Algebra Appl.} {\bf 371} 75-102 (2003).

 \bibitem{ZHANG2011} Farid, F. O.; Wang, Qing-Wen; Zhang, Fuzhen. On the eigenvalues of quaternion matrices. {\em Lin. Multilin. Alg.} {\bf 59}, 4, 451--473 (2011).

 \bibitem{GELFAND1992} Gelfand, I.M.; Retakh, V.S.
A Theory of Noncommutative Determinants and Chracterisitic Functions of Graphs.
{\em Funct. Anal. Appl} {\bf 26} No. 4, 1-20 (1992).

\bibitem{GELFAND2005} Gelfand, I.; Gelfand, S.; Retakh, V.; Lee Wilson, R.
Quasideterminants.
\emph{Adv. Math.} {\bf 193} 56-141 (2005).

\bibitem{GEORGIEV} Georgiev, G.; Ivanov, I.; Mihaylova, M.; Dinkova, T.
\newblock An algorithm for solving a Sylvester quaternion equation.
\newblock {\em Proc. Annual Conf. Rousse Univ. and Sc. Union}, 48, 6.1:35--39, 2009.

\bibitem{HUANG2000}Huang, L.
On two questions about quaternion matrices.
{\em Linear Algebra Appl.} {\bf 318} No.1-3, 79-86 (2000).

\bibitem{HUANGSO2001}
Huang, L.; So, W.
\newblock On left eigenvalues of a quaternionic matrix.
\newblock {\em Linear Algebra Appl.}, 323, No.1-3:105--116, 2001.

\bibitem{HUANGSO2002} Huang, L.; So, W.
\newblock Quadratic formulas for quaternions.
\newblock {\it Appl. Math. Lett.}, 15, No.5:533--540, 2002.

\bibitem{JANOVSKAOPFER2008} Janovsk\'a, D.; Opfer, G.
\newblock Linear equations in quaternionic variables.
\newblock {\it Mitt. Math. Ges. Hamb.}, 27: 223-234 (2008)

\bibitem{JANOVSKAOPFER2010} Janovsk\'a, D.; Opfer, G.
\newblock The classification and the computation of the zeros of quaternionic, two-sided polynomials.
\newblock {\it Numer.
Math.}, 115, No. 1:81--100, 2010.

\bibitem{JOHNSON}
Johnson, R.E.
\newblock On the equation $\chi\alpha =\chi\gamma +\beta$ over an algebraic division ring.
\newblock {\it Bull. Am. Math. Soc.}, 50: 202--207, 1944.

\bibitem{KYRCHEI} Kyrchei, I. I.
Determinantal representations of the Moore-Penrose inverse over the quaternion skew field and corresponding Cramer's rules.
{\em Linear Multilinear Algebra} {\bf 59} No. 4 , 413-431 (2011).

\bibitem{MADSENTORNEHAVE}
Madsen, I.; Tornehave, J.
\newblock {\em From calculus to cohomology: de Rham cohomology and characteristic classes}.
\newblock Cambridge University Press, 1997.

\bibitem{MASSEY}
Massey, William S.
{\em A basic course in algebraic topology.}
Graduate Texts in Mathematics, 127. Springer-Verlag (1991).

\bibitem{MP1}  Mac\'{\i}as-Virg\'os, E.; Pereira-S\'aez, M.J.
\newblock Left eigenvalues of $2\times 2$ symplectic matrices.
\newblock {\it Electron. J. Linear Algebra}, 18:274--280, 2009.

\bibitem{NANSON1900} Nanson, J.
\newblock On certain determinant theorems.
\newblock {\it J. Reine Angew. Math} {\bf 122}: 179--185 (1900).

\bibitem{PUMWAL} Pumpl\"un, S.; Walcher, S.
\newblock On the zeros of polynomials over quaternions.
\newblock {\it Commun. Algebra}, 30(8):4007--4018, 2002.

\bibitem{SO2005}  So, W.
Quaternionic left eigenvalue problem.
{\em Southeast Asian Bull. Math.} {\bf 29} No. 3, 555-565 (2005).

\bibitem{WOOD} Wood, R.M.W.
\newblock Quaternionic eigenvalues.
\newblock{\it Bull. Lond. Math. Soc.}, 17:137--138, 1985.

\bibitem{ZHANG1997}
Zhang, F.
\newblock
Quaternions and matrices of quaternions.
\newblock {\it Linear Algebra Appl.}, 251:21--57, 1997.

\bibitem{ZHANG2007} Zhang, F.
\newblock Ger\v{s}gorin type theorems for quaternionic matrices.
\newblock {\it Linear Algebra Appl.}, 424, No. 1:139--153, 2007.

\end{thebibliography}
}
\end{document}